\documentclass[10pt]{article}
\font\smallit=cmti10
\font\smalltt=cmtt10

\usepackage{amssymb,latexsym,amsmath,epsfig,amsthm} 
\usepackage{framed}
\usepackage{enumitem}
\usepackage{mathtools}
\usepackage[utf8]{inputenc} 
\usepackage{amsrefs}
\usepackage{todonotes}

\makeatletter

\renewcommand\section{\@startsection {section}{1}{\z@}
	{-30pt \@plus -1ex \@minus -.2ex}
	{2.3ex \@plus.2ex}
	{\normalfont\normalsize\bfseries\boldmath}}

\renewcommand\subsection{\@startsection{subsection}{2}{\z@}
	{-3.25ex\@plus -1ex \@minus -.2ex}
	{1.5ex \@plus .2ex}
	{\normalfont\normalsize\bfseries\boldmath}}

\renewcommand{\@seccntformat}[1]{\csname the#1\endcsname. }

\makeatother

\newtheorem{theorem}{Theorem}

\newtheorem{proposition}[theorem]{Proposition}

\newtheorem{claim}[theorem]{Claim}

\theoremstyle{definition}
\newtheorem{definition}[theorem]{Definition}

\newcommand{\NN}{\ensuremath{\mathbb N}}
\newcommand{\ZZ}{\ensuremath{\mathbb Z}}

\newcommand{\CC}{\ensuremath{\mathbb C}}

\newcommand{\fallingfactorial}[2]{#1^{\underline{#2}}}
\newcommand{\risingfactorial}[2]{#1^{\overline{#2}}}

\newcommand{\ve}{\vec {e}}
\newcommand{\vr}{\vec {\rho}}
\newcommand{\vo}{\vec{\omega}}

\newcommand{\ff}{\fallingfactorial}
\newcommand{\rf}{\risingfactorial}

\DeclareMathOperator{\Poly}{\textsc{Poly}}
\DeclareMathOperator{\Rem}{\textsc{Rem}}
\newcommand{\sH}[4]{\Poly_{#1}\!\big( #2 \, \big\lvert\, \substack{#3 \\ #4} \big)}
\newcommand{\sG}[3]{\Rem\!\big( #1 \, \big\lvert \, \substack{#2 \\ #3} \big)}
\newcommand{\sI}[3]{I\big( #1 \, \big\lvert \, \substack{#2 \\ #3} \big)}

\newcommand{\Gzor}{\sG{z}{\vo}{\vr}}
\newcommand{\Hzor}{\sH{m}{z}{\vo}{\vr}}

\newmuskip\pFqskip
\mathchardef\pFcomma=\mathcode`, 
\newcommand*\pFq[5]{%
  \begingroup
  \begingroup\lccode`~=`,
    \lowercase{\endgroup\def~}{\pFcomma\mkern\pFqskip}%
  \mathcode`,=\string"8000
  {}_{#1}F_{#2}\biggl[\genfrac..{0pt}{}{#3}{#4} \,;\, #5\biggr]%
  \endgroup
}

\newcommand*\MeijerG[7]{%
  \begingroup
  \begingroup\lccode`~=`,
    \lowercase{\endgroup\def~}{\pFcomma\mkern\pFqskip}%
  \mathcode`,=\string"8000
  G^{#1,#2}_{#3,#4}\biggl(#5 \,\bigg\lvert\, \genfrac..{0pt}{}{#6}{#7} \biggr)%
  \endgroup
}

\DeclareMathOperator{\ord}{ord_{z=0}} 
\DeclareMathOperator{\adj}{adj} 
\DeclareMathOperator{\PV}{P.V.} 
\DeclareMathOperator{\sgn}{sgn} 

\begin{document}
\begin{center}
	\uppercase{\bf Multidimensional Padé approximation of binomial functions: Equalities}
	\vskip 20pt
{\bf Michael A. Bennett\footnote{supported in part by a Natural Sciences and Engineering Research Council of Canada Discovery Grant}}\\
	{\smallit Department of Mathematics, University of British Columbia, Room 121, 1984 Mathematics Road, Vancouver, BC, Canada}\\
	{\tt bennett@math.ubc.ca}\\
	\vskip 10pt
{\bf Greg Martin\footnote{supported in part by a Natural Sciences and Engineering Research Council of Canada Discovery Grant}}\\
	{\smallit Department of Mathematics, University of British Columbia, Room 121, 1984 Mathematics Road, Vancouver, BC, Canada}\\
	{\tt gerg@math.ubc.ca}\\
	\vskip 10pt
{\bf Kevin O'Bryant\footnote{Support for this project was provided by a PSC-CUNY Award, jointly funded by The Professional Staff Congress and The City University of New York.}}\\
	{\smallit Department of Mathematics, City University of New York, The College of Staten Island and The Graduate Center \\ 2800 Victory Boulevard, Staten Island, NY USA}\\
	{\tt kevin.obryant@csi.cuny.edu}\\
\end{center}
\vskip 20pt
\centerline{\smallit Received: 1/9/21, Revised: 8/3/21, Accepted: 8/12/21, Published: 8/23/21 } 
\vskip 30pt

\begin{quote}
	This article appears in the Ron Graham Memorial Volume of Integers: Electronic Journal of Combinatorial Number Theory. (Integers 21A (2021), Paper No.\ A4, 29 pages). 
\end{quote}

\centerline{\bf Abstract}
\noindent
Let $\omega_0,\dots,\omega_M$ be complex numbers. If $H_0,\dots,H_M$ are polynomials of degree at most $\rho_0,\dots,\rho_M$, and $G(z)=\sum_{m=0}
^M H_m(z) (1-z)^{\omega_m}$ has a zero at $z=0$ of maximal order (for the given $\omega_m,\rho_m$), we say that $H_0,\dots,H_M$ are a 
\emph{multidimensional Padé approximation of binomial functions}, and call $G$ the Padé remainder. We collect here with proof all of the known expressions 
for $G$ and $H_m$, including a new one: the Taylor series of $G$. We  also give a new criterion for systems of Padé approximations of binomial functions to be perfect (a specific sort of independence used in applications).

\pagestyle{myheadings} 
\markright{\smalltt INTEGERS: 21 (2021)\hfill} 
\thispagestyle{empty} 
\baselineskip=12.875pt 
\vskip 30pt


\section{Introduction}
Fix complex functions $f_0,f_1,\dots,f_M$ (all analytic in a neighborhood of 0) and nonnegative integers $\rho_0,\dots,\rho_M$. The set of functions
	 \[X \coloneqq  \left\{ \sum_{m=0}^M H_m(z) f_m(z) : H_m \in \CC[z], \deg(H_m)\leq \rho_m \right\}\]
forms a finite dimensional vector space, and the subsets of functions 
	 \[X_s \coloneqq  \left\{ G \in X : \ord(G)\geq s\right\}\]
with a zero at $z=0$ of order at least $s$ are subspaces. Trivially $X_0 \supseteq X_1 \supseteq X_2 \supseteq \cdots$. Let $\sigma$ be the least integer with $X_\sigma$ having dimension 0, if such $\sigma$ exists. Then $X_{\sigma-1}$ has positive dimension, and the functions in 
$X_{\sigma-1}$ are of particular interest, and are called the {\it Padé remainders  of $f_0,\dots,f_M$}.

The $M=1$ case is the standard tool in numerical analysis known as
Padé approximation~\cite{BakerGravesMorris}, which generalizes Taylor Series. In particular, if $f_0(z)=-1$ identically, and $\rho_1=0$, then
	 \[X=\left\{-H_0(z)+ H_1 \cdot f_1(z) : H_0 \in \CC[z],  \deg(H_0) \leq \rho_0, H_1 \in \CC \right\}.\]
Taking $H_0(z)/H_1$ to be the $\rho_0$-th Taylor polynomial of $f_1(z)$, we find that the Padé remainders are the constant multiples of the Taylor polynomial remainder. 
Letting $\rho_1\geq 0$ leads to rational functions $H_0(z)/H_1(z)$ that approximate $f_1(z)$ at least as well as Taylor polynomials. If $f_1$ has poles near 
0, then this rational approximation is typically much sharper than the Taylor's polynomial approximation.

When $M>1$, we include the adjective ``multidimensional''. This setting has not been exploited as systematically as the $M=1$ case. For a few particular choices of $f_0,\dots,f_M$, there is enough structure that we can work out explicit formulae for the Padé remainders and for the system of Padé approximants, i.e., generating polynomials $H_0,\dots,H_M$. In this paper, we take the binomials $f_m(z) \coloneqq (1-z)^{\omega_m}$ for complex numbers $\omega_0,\dots,\omega_M$, no pair of which has an integer difference. The resulting system of equations was studied by Riemann~\cite{Riemann}, Thue~\cite{Thue}, Siegel~\cite{Siegel}, Mahler~\cite{Mahler}, Baker~\cite{Baker}, Chudnovsky~\cite{Chudnovsky}, Bennett~\cite{Bennett}, and many others, and the use of these Padé approximations for Diophantine analysis is known as the method of Thue-Siegel.

We present in this article our exposition of these classic results on multidimensional Padé approximation of binomial functions.  We combine, and in some cases simplify, the work of Mahler and Jager~\cites{Mahler,Jager}. While there are some original results here, e.g., Theorem~\ref{thm:G}(\ref{thm:G(iv)}) and some cases of Theorem~\ref{thm:perfect}, we see the main value of this work as collating the work of many people over many years with common notation,  complete proofs, and specialization to the choice $f_m(z)$. The results presented in this work are equalities, and so as a check against off-by-one errors, one can implement the various forms given and directly check the equations for randomly chosen parameters. We have done so in Mathematica; a notebook containing these calculations is on the arXiv.

The current work focuses on various expressions for the Padé remainders and approximants. In subsequent works, we will provide new bounds, both archimedian and non-archimedian, on the size of the approximant polynomials $H_0,\dots,H_M$ and on the Padé remainder, and will exploit  those bounds to give new irrationality measures for some numbers  of the form $(a/b)^{s/n}$.

\section{Statement of Results}

Let $M$ be a nonnegative integer. Consider $\vo\coloneqq \langle \omega_0, \omega_1,\dots,\omega_M\rangle$, a vector of $M+1$ distinct complex numbers, no 
pair of which has a difference that is an integer, and $\vr \coloneqq  \langle \rho_0, \dots,\rho_M \rangle$, a vector of $M+1$ nonnegative integers (typically not 
distinct). We index the vectors $\vo\in\CC^{M+1},\vr\in \NN^{M+1}$ with $0,1,\dots,M$; for example, the $0$-th  coordinate of $\vr$ is $\rho_0$ and the $M$-th coordinate is $\rho_M$. We will only consider $M$, $\vo$, $\vr$ satisfying these constraints. Two fundamental parameters are
	\[
	\sigma = \sigma(\vr) \coloneqq  \sum_{m=0}^M (\rho_m+1), \quad\text{ and }\quad
	\vr\,! \coloneqq  \prod_{m=0}^M \rho_m!.
	\]

Some notation used in Theorem~\ref{thm:consistency} is both standard and uncommon; we give definitions in the next section. When we add a scalar to a vector, we mean that the scalar is added to each coordinate, such as $\vr+1= \langle \rho_0+1,\rho_2+1,\dots,\rho_M+1\rangle$. When we delete the $m$-th coordinate, reducing the length of the vector by 1, we use a ``$\star m$'' exponent, such as
	\[\vo^{\star m} = \langle \omega_0, \dots, \omega_{m-1},\omega_{m+1}, \dots, \omega_M\rangle.\]
The standard basis vectors are denoted $\ve_0,\ve_1,\dots,\ve_M$.

\begin{theorem}\label{thm:consistency}
Let $\vr$ and $\vo$ be fixed vectors as above. 
\begin{enumerate}[label=(\roman*)]
\item \label{thm:consistency, existence}
    (Existence) There are polynomials $H_m$ in $z$ of degree at most $\rho_m$, with at least one $H_m$ not identically 0, and with
	\[ G(z)\coloneqq \sum_{m=0}^M H_m(z)(1-z)^{\omega_m}\] 
having a zero of order at least $\sigma-1$ at $z=0$. 
\item \label{thm:consistency, uniqueness}
    (Uniqueness) For such $G(z)$, the function $G(z)$ necessarily has a zero of order exactly $\sigma-1$ at $z=0$, and furthermore the polynomials $H_m(z)$ are uniquely determined given the additional constraint that 
\[\lim_{z\to0} \frac{G(z)}{z^{\sigma-1}} = \frac{1}{(\sigma-1)!}.\]
Each $H_m(z)$ has degree exactly $\rho_m$. There is no $\alpha\in\CC$ with
    \[ H_0(\alpha)  = \dots = H_M(\alpha) = 0.\]
\item \label{thm:consistency, domain}
    (Domain) $G(z)$ is analytic on $\CC \setminus[1,\infty)$.
\end{enumerate}
\end{theorem}

Theorem~\ref{thm:consistency} allows us to make the following definition of Padé approximants and remainders.

\begin{definition}\label{def:Pade notation}
Let $\vr$ and $\vo$ be fixed vectors as above. The $M+1$ {\it Padé approximants} $\Hzor$ (with $0\le m \le M$) are the polynomials with degrees $\rho_m$, and with {\it Padé remainder}
    \[\Gzor \coloneqq \sum_{m=0}^M \Hzor \, (1-z)^{\omega_m}\] 
both having a zero of order $\sigma-1$ at $z=0$, and satisfying
    \[\lim_{z\to0} \frac{\Gzor}{z^{\sigma-1}} = \frac{1}{(\sigma-1)!}.\]
\end{definition}

In Proposition~\ref{Claim:symmetry},  we draw attention to some obvious symmetries, immediate from Theorem~\ref{thm:consistency}, whose proofs we do not spell out.

\begin{proposition}[Permutation and Shift Symmetry]\label{Claim:symmetry}
If $\pi$ is any permutation of $0,1,\dots,M$, then
    \[    \Gzor = \sG{z}{\langle \omega_0,\omega_1, \dots, \omega_M \rangle}{\langle \rho_0, \rho_1, \dots, \rho_M \rangle} 
        =  \sG{z}{\langle \omega_{\pi(0)},\omega_{\pi(1)}, \dots, \omega_{\pi(M)}\rangle}{\langle \rho_{\pi(0)}, \rho_{\pi(1)}, \dots, \rho_{\pi(M)} \rangle}
    \]
and 
    \[ \Hzor = \sH{m}{z}{\langle \omega_0,\omega_1, \dots, \omega_M \rangle}{\langle \rho_0, \rho_1, \dots, \rho_M \rangle} 
        =  \sH{\pi^{-1}(m)}{z}{\langle \omega_{\pi(0)},\omega_{\pi(1)}, \dots, \omega_{\pi(M)}\rangle}{\langle \rho_{\pi(0)}, \rho_{\pi(1)}, \dots, \rho_{\pi(M)} \rangle}.
    \]
For any $\alpha$, we have
	\[
	(1-z)^\alpha\Gzor = \sG{z}{\alpha +\vo}{\vr}
	\quad \text{and}\quad
	\Hzor = \sH{m}{z}{\alpha +\vo}{\vr}.
	\]
\end{proposition}

The purpose of the current work is to collect together various explicit formulae for the Padé remainder $\Gzor$ and the Padé approximants $\Hzor$, in a common notation, and with complete proofs. Formulae for the Padé remainder are given in Theorem~\ref{thm:G}, and formulae for the Padé approximants are given in Theorem~\ref{thm:H}.

\begin{theorem}[Forms for the Padé Remainder]\label{thm:G}
The following five expressions give $\Gzor$.
\begin{enumerate}[label=(\roman*)]
\item \label{thm:G(i)}
    The Padé remainder $\Gzor$ is given by the iterated integral
	\[\frac{(1-z)^{\omega_0}}{\vr\,!} \int_0^z \int_0^{t_1} \int_0^{t_2} \cdots \int_0^{t_{M-1}}
	\mathcal{G}(z,t_1,t_2,\dots,t_M)\,dt_M\cdots dt_3\,dt_2\,dt_1,\]
where
	\[\mathcal{G}(t_0,t_1,\dots,t_M)
		= t_M^{\rho_M}
			\left(\prod_{h=1}^{M} \left( \frac{t_{h-1}-t_{h}}{1-t_{h}}\right)^{\rho_{h-1}}\right)
			\left(\prod_{h=1}^M (1-t_h)^{\omega_h-\omega_{h-1}-1} \right).\]

\item \label{thm:G(ii)}
    The Padé remainder $\Gzor$ is given by the $M$-dimensional integral
        \[z^{\sigma-1} \frac{(1-z)^{\omega_0}}{\vr\,!} \int_{[0,1]^M}
                          U_M^{-1}\prod_{h=1}^M \frac{U_h^{1+\rho_h}}{(1-zU_h)^{1-\omega_h+\omega_{h-1}}}
                         \left( \frac{1-u_h}{1-z U_h} \right)^{\rho_{h-1}} 
                         \,d\vec u,
        \]
where $U_m = \prod_{h=1}^m u_h$.

\item \label{thm:G(iii)}
    The Padé remainder $\Gzor$ is the contour integral
	\[\frac{(-1)^{\sigma-1}}{2\pi i} \int_\gamma  (1-z)^\xi  \prod_{k=0}^M \frac{1}{\ff{(\xi-\omega_ k)}{\rho_ k+1}} \,d\xi,\]
where $\gamma$ is any simple positively oriented contour enclosing all $\sigma$ of the complex numbers $\omega_m+r$  ($0 \leq m \leq M, 0\leq r \leq 
\rho_m$).

\item \label{thm:G(iv)} The Maclaurin series for $\Gzor$ is
	\[
	\sum_{n=0}^\infty (-1)^n \sum_{m=0}^M \frac{1}{\rho_m!} 
			\sum_{r=0}^{\rho_m} \binom{\rho_m}{r} \,
							\frac{(-1)^r \ff{(\omega_m+r)}{n}}{\prod_{\substack{k=0\\ k\neq m}}^M \rf{(\omega_k-\omega_m-r)}{\rho_k+1}}
			\, \frac{z^n}{n!},
	\]
    which converges for $|z|<1$.

\item \label{thm:G(v)}
    Finally, $\Gzor$ is the special value of Meijer's $G$-function given by
   	\[ \MeijerG{M+1}{0}{M+1}{M+1}{1-z}{\vo+\vr+1}{\vo}.\]
\end{enumerate}
\end{theorem}

In addition to the formulae for $\Gzor$  given in Theorem~\ref{thm:G}, we note that
      \[\Gzor = \sum_{m=0}^M \Hzor \, (1-z)^{\omega_m},\]
and so any formula for $\Hzor$ generates a formula for $\Gzor$.  Theorem~\ref{thm:H}  gives a number of useful representations of $\Hzor$.

\begin{theorem}[Forms for the Padé Approximants]\label{thm:H}
The following five expressions give $\Hzor$.
\begin{enumerate}[label=(\roman*)]
\item \label{thm:H(i)}
    Let $\gamma_m$ be a simple positively oriented contour enclosing all $\rho_m+1$ of the complex numbers $\omega_m+r$ ($0\leq r \leq \rho_m$) and 
none of $\omega_k+r$  ($0 \leq k \leq M, k \neq m, 0\leq r \leq \rho_k$). Then $\Hzor$ is given by
	\[\frac{(-1)^{\sigma-1}}{2\pi i} \int_{\gamma_m} (1-z)^{\xi-\omega_m}  \prod_{k=0}^M \frac{1}{\ff{(\xi-\omega_ k)}{\rho_ k+1}} \,d\xi.\]

\item \label{thm:H(ii)}
    The Padé approximant $\Hzor$ is equal to
	\[\frac{1}{\rho_m!}\,\sum_{r=0}^{\rho_m} (z-1)^r \binom{\rho_m}{r} 
		\prod_{\substack{k=0\\ k\neq m}}^M \frac{1}{\rf{(\omega_k-\omega_m-r)}{\rho_ k+1}} .\]

\item \label{thm:H(iii)}
    For $M\geq1$, the Padé approximant $\Hzor$ is the $M$-fold iterated integral
	\[
	\frac{Q_m}{\vr\,!} \, \int_{(G)} T_m^{-\omega_m-1} \bigg( \prod_{\substack{k=0\\ k\neq m}}^M  t_k^{\omega_k}(1+t_k)^{\rho_k}\bigg)
		\left(1 - (-1)^M \,\frac{1-z}{T_m} \right)^{\rho_m} \,d\vec{t},
	\]
where $\int_{(G)} \cdots d\vec{t}$ integrates each of $t_0,\dots,t_M$ (except $t_m$) counterclockwise on the unit circle from $-\pi$ radians to $\pi$ radians (i.e., the principal value),
	\[
	Q_m 
		\coloneqq  \prod_{\substack{k=0 \\ k \neq m}}^M \frac{1}{2i\sin(\pi(\omega_k-\omega_m))}, \quad \text{and} \quad
	T_m
		\coloneqq  \prod_{\substack{k=0 \\ k \neq m}}^M t_k .
	\]

\item \label{thm:H(iv)}
    The Padé approximant is a scaled generalized hypergeometric function:
	\[
	\frac{1}{\rho_m!} \bigg( \prod_{\substack{k=0 \\ k \neq m}}^M \frac{1}{\rf{(\omega_k-\omega_m)}{\rho_k+1}}\bigg)
	\,\pFq{M+1}{M}{\omega_m-\vo-\vr}{(1+\omega_m - \vo)^{\star m}}{1-z}.
	\]

\item \label{thm:H(v)}
    Set $W\coloneqq W(m,k)=\omega_k-\omega_m$, and define $C_{m,k,r}$ by
    \[C_{m,k,r} \coloneqq  \dbinom{\rho_k}{r},\]
if $m=k$, by
    \[C_{m,k,r} \coloneqq
          (-1)^{\rho_k+1} \dbinom{r}{\rho_k}^{-1} \dfrac{\Gamma(r+1)}{\Gamma(r+1-W)} \dfrac{\Gamma(r-\rho_k-W)}{\Gamma(r-\rho_k+1)}\]
if $m\neq k$ and $\rho_k<r$, and  by
    \[C_{m,k,r} \coloneqq (-1)^r \dbinom{\rho_k}{r} \dfrac{\Gamma(r+1)}{\Gamma(r+1-W)} \dfrac{\Gamma(\rho_k-r+1)}{\Gamma(\rho_k-r+1+W)}\dfrac{\pi}{\sin(\pi W)}\]
if $m\neq k$ and $\rho_k \geq r$.
Then, we have
       \[
       \Hzor = \frac{1}{\vr\,!} \sum_{r=0}^{\rho_m} (z-1)^r \prod_{k=0}^M C_{m,k,r}.
       \]
\end{enumerate}                          
\end{theorem}

Theorem~\ref{thm:perfect} precisely states that notion that the approximants for nearby $\vr$ are independent. This property is referred to as  ``perfect approximation'', and relies mostly on $\deg(\sH{m}{z}{\vo}{\vr}) = \rho_m$ and $\ord(\Gzor)=\sigma-1$. Recall that our $M+1$ dimensional vectors have coordinates indexed from 0 through $M$.

\begin{theorem}[Approximants are Perfect]\label{thm:perfect}
Fix $\vr \in \NN^{M+1}$ and $\vec\epsilon_0, \vec\epsilon_1, \dots, \vec\epsilon_M \in \ZZ^{M+1}$ with each $\vr+\vec\epsilon_k$ having nonnegative coordinates, and denote the $j$-th coordinate of $\vec\epsilon_i$ as $\vec\epsilon_{i,j}$. Let $S$ be maximum of $\sum_{i=0}^M \vec\epsilon_{i,\beta(i)}$ taken over all permutations $\beta$ of $0,1, \dots, M$, and let $T$ be the minimum of $\sum_{j=0}^M \vec\epsilon_{i,j}$ taken over $0\leq i \leq M$. Suppose the following two conditions are satisfied:
\begin{enumerate}[label=(\roman*)]
  \item There is a unique permutation $\alpha$ of $0,1,\dots,M$ with $S=\sum_{i=0}^M \vec\epsilon_{i,\alpha(i)}$;
  \item $T+M = S$.
\end{enumerate}
Then the $(M+1) \times (M+1)$ matrix whose $(k,m)$ coordinate is the polynomial $\sH{m}{z}{\vo}{\vr+\vec\epsilon_k}$ has determinant
  \[
  C z^{\sigma(\vr)+T-1},
  \]
where $C$ does not depend on $z$.
\end{theorem}

The most startling aspect of Theorem~\ref{thm:perfect} is that $\vo$ plays no role in the hypotheses nor in the conclusion.

We note that (in Theorem~\ref{thm:perfect}) with $\vec\epsilon_k = \ve_k$ one has $T=1, S=M+1$, and the conditions in Theorem~\ref{thm:perfect} are satisfied. This recovers a result stated and used by Mahler, Chudnovsky and Bennett~\cites{Bennett,Chudnovsky,Mahler}. If one takes $I_k \subseteq \{ 0, 1, \dots, k-1\}$ and $\vec\epsilon_k = \ve_k + \sum_{i\in I_k} \ve_i$, one recovers a result of Jager~\cite{Jager}. Our result covers many more examples than we found in the literature, but it is not exhaustive. 

\section{More Notation}
We denote the rising and falling factorials as
	\begin{align*}
	          \rf{x}{r} &\coloneqq  x\cdot \rf{(x+1)}{r-1}=x\cdot (x+1)\cdot (x+2) \cdots (x+r-1),\\
	          \ff{x}{r} &\coloneqq  x\cdot \ff{(x-1)}{r-1}=x\cdot(x-1)\cdot (x-2) \cdots (x-r+1),
          \end{align*}
for positive integers $r$, and define $\rf{x}{0}=\ff{x}{0}=1$.
We use the following trivial identities without comment (provided $x-r+1 \not\in\{0,-1,-2,\dots\}$):
	\[\ff{x}{r} = \frac{\Gamma(x+1)}{\Gamma(x-r+1)}, \quad
	\ff{x}{r} = \rf{(x-r+1)}{r} = (-1)^r \rf{(-x)}{r}, \]
and typically choose to eliminate ratios of $\Gamma$ functions in preference for the more computationally friendly rising and falling factorials.
All of our functions will be analytic in a complex neighborhood of $z=0$. We use $\deg(f(z))$ to be the degree of $f$, which is $\infty$ if $f$ is not a 
polynomial. We use $\ord(f(z))$ to denote the order of the zero of $f$ at $z=0$, and we use $O(z^k)$ to denote a function that has a zero at $z=0$ of order at 
least $k$.

We shall briefly encounter 
the generalized hypergeometric function (defined for $|z|<1$, $q<p$, and appropriate integers $a_i,b_i$)
	\[
	\pFq{p}{q}{a_1, a_2, \dots, a_p}{b_1,b_2,\dots,b_q}{z} 
		= \sum_{n=0}^\infty \frac{\rf{a_1}{n} \rf{a_2}{n} \cdots \rf{a_p}{n}}{\rf{b_1}{n} \rf{b_2}{n} \cdots \rf{b_q}{n}} \, \frac{z^n}{n!},
	\]
and also the Meijer $G$-function~\cite{NISTonline} 
    $$\MeijerG{m}{n}{p}{q}{z}{a_1,a_2, \dots, a_p}{b_1,b_2,\dots,b_q}$$ 
(defined for natural numbers $m,n,p,q$, provided $m\leq q$ and $n\leq p$, although we only encounter it in this work with $n=0, m=p=q=M+1$), defined by
	\[
		\frac{1}{2\pi i} \int_C \frac{\prod_{k=1}^m \Gamma(s+b_k) \prod_{k=1}^n \Gamma(1-a_k-s)}
            					    {\prod_{k=n+1}^p \Gamma(s+a_k) \prod_{k=m+1}^q \Gamma(1-b_k-s)} \, z^{-s}\,ds,
	\]
where $C$ is an infinite contour that separates the poles of $\Gamma(1-a_k-s)$ from those of $\Gamma(b_k+s)$; the particular contour required for 
convergence varies depending on $m,n,p,q,z$.

\section{Claims and Proofs}
It is at least plausible that there are polynomials $H_0,\dots,H_{M}$ with degrees $\rho_0,\dots,\rho_M$ and
	\begin{equation}
	G(z)\coloneqq \sum_{m=0}^M H_m(z) (1-z)^{\omega_m} = \frac{z^{\sigma-1}}{(\sigma-1)!} +O(z^{\sigma}),
	\end{equation}
where $O(z^\sigma)$ refers to $z\to0$. After all, the polynomials have a total of $\sigma$ coefficients, and we may choose them so that $G(z)$ has a zero at $z=0$ of order $\sigma-1$, and the first nonzero coefficient in the power series expansion of $G(z)$ is according to our choosing. Establishing this rigorously is the point to our first claims.

In all the Claims in this section, we assume that $M$ is a nonnegative integer, and that $0\leq m \leq M$. We assume that $\vr=\langle \rho_0,\dots,\rho_M\rangle $ is vector of $M+1$ nonnegative integers, and that $\vo=\langle \omega_0,\dots,\omega_M\rangle $ is a vector of $M+1$ distinct complex numbers, no two of which have a difference that is an integer. Both $\vr$ and $\vo$ (and vectors derived from them) are indexed 0 through $M$.

\subsection{Existence and Uniqueness}

The following claim is used implicitly frequently throughout this work.
\begin{claim}\label{Claim:G not 0}
For any polynomials $H_m(z)$ (not all zero), the sum
	\[G(z) \coloneqq  \sum_{m=0}^M H_m(z) (1-z)^{\omega_m}\]
is not identically 0.
\end{claim}

\begin{proof}
Since no two $\omega_i$ have difference that is an integer, there is a unique $k$ with
	\[\omega_k + \deg(H_k) = \max\{ \omega_i + \deg(H_i) : H_i \not =0 \}. \]
Then
	\[\lim_{z \to - \infty} \frac{G(z)}{H_k(z)(1-z)^{\omega_k}} 
		= 1 + \sum_{\substack{m=0 \\ m\neq  k}}^M \lim_{z\to -\infty} \frac{H_m(z)(1-z)^{\omega_m}}{H_k(z)(1-z)^{\omega_k}} = 1.\]
Consequently, $G$ cannot be identically 0.
\end{proof}

\begin{claim}\label{Claim:ord(G) not small}
There are polynomials $H_0(z),\dots,H_M(z)$ of degrees at most $\rho_0,\dots,\rho_M$, respectively, not all identically 0, such that
	\[\ord\left(\sum_{m=0}^M H_m( z ) (1-z)^{\omega_m} \right) \geq \sigma-1. \]
\end{claim}

\begin{proof}
Consider polynomials $H_0(z),\dots,H_M(z)$ of degrees $\rho_0,\dots,\rho_M$ with unknown coefficients, a total of $\sigma$ unknowns. 
Recall Newton's Binomial Theorem: for $|z|<1$ and any complex $\omega$, we have
	\[ (1-z)^\omega = \sum_{i=0}^\infty (-1)^i \frac{\ff{\omega}{i}}{i!}\, z^i.\]
Considering the coefficient of $z^j$, for $0\leq j \leq \sigma-2$, on both sides of the desired equality
	\[
	\sum_{m=0}^M H_m(z) \sum_{i\geq 0} (-1)^i \frac{\ff{{\omega_m}}{i}}{i!}\, z^i = O(z^{\sigma-1})
	\]
yields a homogeneous linear equation in the unknowns, a total of $\sigma-1$ equations. By linear algebra, there is a choice of the $\sigma$ unknowns, not all zero, which 
satisfies all of the equations. In other words, there are polynomials $H_0(z),\dots,H_M(z)$ (not all zero) with degrees at most $\rho_0,\dots,\rho_M$, such 
that
	\[\sum_{m=0}^M H_m(z) (1-z)^{\omega_m}\]
has a zero of order at least $\sigma-1$ at $z=0$. 
\end{proof}

Claim~\ref{Claim:ord(G) not small} establishes Theorem~\ref{thm:consistency}\ref{thm:consistency, existence}.

The next claim is slightly stronger than the $M=0$ case of Theorem~\ref{thm:consistency}, in that explicit formulae are given, and is used as a base case for subsequent induction arguments.
\begin{claim} \label{Claim:M=0}
The $M=0$ Padé approximant and remainder are given by the formulae 
\(\displaystyle \sH{0}{z}{\langle \omega_0 \rangle}{\langle \rho_0 \rangle} = \frac{z^{\rho_0}}{\rho_0!}\), and 
\(\displaystyle \sG{z}{\langle \omega_0 \rangle}{\langle \rho_0 \rangle} = \frac{z^{\rho_0}}{\rho_0!} (1-z)^{\omega_0}.\)
\end{claim}

\begin{proof}
We need to show that the only nonzero polynomials $H_0$ with $\ord(H_0(z)(1-z)^{\omega_0})\geq \sigma-1$ and degree at most $\rho_0$ are 
$H_0(z)=Cz^{\rho_0}$. First, observe that $\sigma=\rho_0+1$. As $\ord((1-z)^{\omega_0})=0$, we know that $\ord(H_0(z)(1-z)^{\omega_0})=\ord(H_0)$. 
That is, $H_0$ must be a nonzero polynomial with $\ord(H_0)\geq \rho_0$ and $\deg(H_0)\leq \rho_0$. The only candidates are $\sH{0}{z}{\langle \omega_0 \rangle}{\langle \rho_0 \rangle} = C z^{\rho_0}$ and $\sG{z}{\langle \omega_0 \rangle}{\langle \rho_0 \rangle} = C z^{\rho_0}(1-z)^{\omega_0}$.

Now, observe that 
	\[
	C = \lim_{z\to0} \frac{Cz^{\rho_0}(1-z)^{\omega_0}}{z^{\rho_0}} = \frac{1}{(\sigma-1)!} = \frac{1}{\rho_0!}.
	\] 
Thus, Theorem~\ref{thm:consistency}\ref{thm:consistency, uniqueness} is proved in the $M=0$ case, and the values of $\Hzor$ and $\Gzor$ are as claimed here.
\end{proof}

\begin{claim}\label{Claim:ord(G) not large}
If $\deg(H_m(z))\leq \rho_m$, and some $H_m\neq 0$, then
	\[\ord\left(\sum_{m=0}^M H_m(z) (1-z)^{\omega_m}\right) \leq \sigma -1.\]
\end{claim}

\begin{proof}
Suppose $M=0$. With $H_0$ a nonzero polynomial with degree at most $\rho_0$, we have
	\[
	\ord\left( H_0(z) (1-z)^{\omega_0}\right) = \ord\left( H_0(z) \right) \leq \deg(H_0) \leq \rho_0 = \sigma-1.\]
So, the claim holds for $M=0$.

Assume the claim is false, and let $M$ be the smallest positive integer for which this claim does not hold, and let $\rho_0$ correspond to the first counterexample: that is, for any $\vo,\vr$ that has a smaller $M$, or the same $M$ but smaller $\rho_0$, the claim holds. Let
	\[
	G(z) \coloneqq  \sum_{m=0}^M H_m(z) (1-z)^{\omega_m} 
	\]
be a counterexample, i.e., $\ord(G) \geq \sigma$.  As multiplying by $(1-z)^{-\omega_0}$ does not change $\ord(G(z))$, we may assume that $\omega_0=0$. 

If $\rho_0=0$, so that $H_0(z)$ is a constant, we have
	\begin{align*}
	\frac{d}{dz} G(z) 
		&=\tfrac{d}{dz}H_0(z) + \sum_{m=1}^M \tfrac{d}{dz}H_m(z) (1-z)^{\omega_m} \\
		&= \sum_{m=1}^M \left(H'_m(z)(1-z)-H_m(z)\omega_m \right) (1-z)^{\omega_m-1}.
	\end{align*}
Note that $\deg(H'_m(z)(1-z)-H_m(z)\omega_m) \leq \deg(H_m) \leq \rho_m$, for $1\leq m \leq M$. Thus $\tfrac{d}{dz}G(z)$ has a smaller $M$ and the 
same $\rho_m$. By assumption on $G(z)$, 
$$\ord(\tfrac {d}{dz}G(z)) \geq \sigma-1,$$ but by our assumption of the minimality of $G(z)$, we know that
	\[
	\ord\left( \tfrac{d}{dz} G(z) \right) \leq \sigma-2.
	\]
This contradiction shows that $\rho_0\neq 0$. But even in the case that $\rho_0>0$,
	\begin{align*}
	\frac{d}{dz} G(z) 
		&=H_0'(z)+\sum_{m=1}^M \tfrac{d}{dz}H_m(z) (1-z)^{\omega_m} \\
		&= H_0'(z)+\sum_{m=0}^M \left(H'_m(z)(1-z)-H_m(z)\omega_m \right) (1-z)^{\omega_m-1}.
	\end{align*}
As above, our assumption on the minimality of $\rho_0$, as $\deg(H'_0)=\deg(H_0)-1$, implies a contradiction.
\end{proof}

The proof of the next claim establishes the rest of Theorem~\ref{thm:consistency}\ref{thm:consistency, uniqueness}, and justifies Definition~\ref{def:Pade notation}.

\begin{claim}
Suppose that $H_m$ (with $0\leq m \leq M$) are polynomials with degree at most $\rho_m$, and that $G(z):= \sum_{m=0}^M H_m(z)(1-z)^{\omega_m}$ has a zero of order at least $\sigma-1$ at $z=0$. Then $G(z)$ has an order of exactly $\sigma-1$ at $z=0$. Suppose further that 
    \[\lim_{z\to0} \frac{G(z)}{z^{\sigma-1}} = \frac{1}{(\sigma-1)!}.\]
Then $G$ and $H_m$ are uniquely determined by these constraints.
The polynomial $H_m(z)$ has degree exactly $\rho_m$, and there is no $\alpha\in\CC$ with $H_0(\alpha)=\dots=H_m(\alpha)=0$. 
\end{claim}

\begin{proof}
By Claims~\ref{Claim:ord(G) not small} and~\ref{Claim:ord(G) not large}, we can take $\ord(G(z))$ to be at least $\sigma-1$, and can never have it be larger than $\sigma-1$, so there are polynomials $H_m$ 
with
	\[G(z) = \sum_{m=0}^M H_m(z) (1-z)^{\omega_m}  = C z^{\sigma-1} + O(z^\sigma).\]
By multiplying through by a constant, we can take
	\[C = \frac{1}{(\sigma-1)!}.\]

If both $G_1(z)$ and $G_2(z)$ have this form, then their difference would have a zero of order greater than $\sigma-1$, and by Claim~\ref{Claim:ord(G) not large} 
this is not possible unless $G_1(z)-G_2(z)$ is identically 0. By Claim~\ref{Claim:G not 0}, however, this is only possible if all of the polynomials are identically 
0. That is, only if $G_1(z)=G_2(z)$. Thus, $G$ is uniquely defined and the definition of $\Gzor$ is justified.

If
	\[\Gzor = \sum_{m=0}^M H_m(z) (1-z)^{\omega_m} = \sum_{m=0}^M B_m(z) (1-z)^{\omega_m}\]
for polynomials $H_m,B_m$ of degree at most $\rho_m$, then
	\[0 =  \sum_{m=0}^M (H_m(z) - B_m(z)) (1-z)^{\omega_m}.\]
But by Claim~\ref{Claim:G not 0}, this implies that $H_m(z)=B_m(z)$. Thus, $H_m$ is uniquely defined and the definition of $\Hzor$ is justified.

Suppose that $\Hzor$ has degree strictly less than $\rho_m$, which in particular means that $\rho_m\geq1$. Let $\ve_m$ be the $M+1$-dimensional unit vector in the $m$-th coordinate direction. Then $\sG{z}{\vo}{\vr-\ve_m}$ is a constant multiple of $\Gzor$, which has a zero of order $\sigma(\vr)-1>\sigma(\vr-\ve_m)-1$, contradicting Claim~\ref{Claim:ord(G) not large}.

Finally, if $H_m(\alpha)=0$ for $0\leq m \leq  M$,  then  $H_m(z)/(z-\alpha)$ are polynomials with degree $\rho_m-1$, and $G(z)/(z-\alpha)$ has a zero of order $\sigma(\vr)-1$ at $z=0$, contradicting the uniqueness of  $G$.
\end{proof}

\subsection{Respectful Differential Operators}
\begin{claim}[Differentiation To Reduce $\rho$]\label{Claim:Differentials}
Define the operators
	\[ d_{\omega} \coloneqq  (1-z)^{\omega+1}\left(\frac{d}{dz}\right) (1-z)^{-\omega} .\]
If $\rho_i>0$, then $d_{\omega_i}$ reduces $\rho_i$ by 1 and increases $\omega_i$ by 1, i.e., if $\rho_i>0$, then
	\[ d_{\omega_i} \Gzor = \sG{z}{\vo + \ve_i }{ \vr - \ve_i}.\]
If $\rho_i=0$, then $d_{\omega_i}$ eliminates the $i$-th coordinates of $\vo$ and $\vr$, i.e., if $\rho_i=0$, then
	\[d_{\omega_0} \Gzor = \sG{z}{\vo^{\star i}}{ \vr^{\star i} }.\]
Consequently, for any $\rho_0$,
	\[  (1-z)^{\omega_0+\rho_0+1} \left(\frac{d}{dz}\right)^{\rho_0+1} (1-z)^{-\omega_0} \Gzor =
                \sG{z}{\langle \omega_1, \dots, \omega_M \rangle }{ \langle \rho_1, \dots, \rho_M \rangle}. \]
\end{claim}

\begin{proof}
As $d_{\omega}$ is linear and
    \[\Gzor =\sH{i}{z}{\vo}{\vr} (1-z)^{\omega_i}+ \sum_{\substack{m=0 \\ m \neq i}}^M \Hzor (1-z)^{\omega_m},\]
we can assess the impact of $d_{\omega_i}$ on the two pieces separately. First, 
    \begin{align*}
      d_{\omega_i} \sH{i}{z}{\vo}{\vr} (1-z)^{\omega_i} 
          &=  (1-z)^{\omega_i+1} \left( \frac{d}{dz} \right) (1-z)^{-\omega_i} \cdot \sH{i}{z}{\vo}{\vr}(1-z)^{\omega_i} \\
          &= \left(\frac{d}{dz} \sH{i}{z}{\vo}{\vr} \right) (1-z)^{\omega_i+1} .
    \end{align*}
This is $0$ if $\rho_i=0$, and if $\rho_i>0$ it has the form $P_i(z)(1-z)^{\omega_i+1}$ with $P_i$ a polynomial of degree $\rho_i-1$. The other piece is more involved (for the sake of the margins, we let $H(z)\coloneqq \Hzor$ in the following displayed equations):
    \begin{align*}
      d_{\omega_i} \sum_{\substack{m=0 \\ m \neq i}}^M & \Hzor \, (1-z)^{\omega_m} \\
            &= (1-z)^{\omega_i+1} \left( \frac{d}{dz} \right) (1-z)^{-\omega_i} \sum_{\substack{m=0 \\ m \neq i}}^M H(z) (1-z)^{\omega_m} \\
            &= (1-z)^{\omega_i+1} \sum_{\substack{m=0 \\ m \neq i}}^M 
                        \frac{d}{dz} H(z) (1-z)^{\omega_m-\omega_i} \\
            &= (1-z)^{\omega_i+1} \sum_{\substack{m=0 \\ m \neq i}}^M 
                        (1-z)^{\omega_m-\omega_i} \tfrac{d}{dz} H(z) 
                      - H(z) (\omega_m-\omega_i)(1-z)^{\omega_m-\omega_i-1} \\
            &= (1-z)^{\omega_i+1} \sum_{\substack{m=0 \\ m \neq i}}^M  
                        \left( (1-z) \tfrac{d}{dz} H(z) -  (\omega_m-\omega_i) H(z) \right) (1-z)^{\omega_m-\omega_i-1}\\
            &=\sum_{\substack{m=0 \\ m \neq i}}^M \left( (1-z) \tfrac{d}{dz} \Hzor 
            		- (\omega_m-\omega_i) \Hzor  \right) (1-z)^{\omega_m}.
    \end{align*}
This has the form
    \[ \sum_{\substack{m=0 \\ m \neq i}}^M P_m(z) (1-z)^{\omega_m}\]
with $P_m$ a polynomial of degree at most $\rho_m$. To wit, $d_{\omega_i}\Gzor$ has the correct form to be $\sG{z}{\vo + \ve_i }{ \vr - \ve_i}$ if $\rho_i>0$, and the correct form to be $\sG{z}{\vo^{\star i}}{\vr^{\star i}}$ if $\rho_i=0$. By our earlier uniqueness result, it remains only to check that $d_{\omega_i}\Gzor$ has a zero (at $z=0$) of order one less than $\Gzor$ and the correct scaling. These are both clear, as $(1-z)^{\omega_i+1}$ and $(1-z)^{-\omega_i}$ have no zero at $z=0$, and the $\tfrac{d}{dz}$ reduces the order of the zero by one and the scaling coefficient is multiplied by $\sigma-1$.

The last sentence of Claim~\ref{Claim:Differentials} is now immediate, as the product of operators telescopes
	\[
	d_{\omega_0+\rho_0}\cdots d_{\omega_0+1}d_{\omega_0} 
		= (1-z)^{\omega_0+\rho_0+1} \left(\frac{d}{dz}\right)^{\rho_0+1} (1-z)^{-\omega_0}.
	\]
\end{proof}

The previous claim establishes $\Gzor$ as the solution of a differential equation (henceforth DE), which we make explicit next. Then, we solve the DE to express $\Gzor$ as 
an $M$-fold iterated integral.

\begin{claim}\label{Claim:G as DE}
Let $D_0,\dots,D_M$ be the operators
	\[	
	D_i \coloneqq  (1-z)^{\omega_i+\rho_i+1} \left(\frac{d}{dz}\right)^{\rho_i+1} (1-z)^{-\omega_i}
	\]
With this notation, $G(z)=\Gzor$ is the unique analytic solution to the differential equation
	\[
	D_{M-1} \cdots D_1 D_0 G(z) = 
				\frac{z^{\rho_M}}{\rho_M!} \, (1-z)^{\omega_M}
	\]
with initial conditions
	\[ G^{(\sigma-1)}(0) = 1, \qquad G^{(m)}(0) = 0, \qquad (0\leq m \leq \sigma-2).\]
\end{claim}

\begin{proof}
That $\Gzor$ satisfies the DE is a consequence of Claim~\ref{Claim:Differentials}, and the initial conditions are part of the definition of $\Gzor$.

As for uniqueness, suppose that $G(z) = \sum_{i=0}^\infty g_i z^i$. Observe that the initial conditions force 
	\[g_{\sigma-1} = \frac{1}{(\sigma-1)!} , \qquad 
		g_i = 0,\quad (0\leq i \leq \sigma-2).\]
The DE then forces the value of $g_i$ for $i\geq \sigma$. 
\end{proof}

It would be interesting to use the proof of the above claim to work out the full power series of $\Gzor$.

\subsection{Iterated Integrals}
Claim~\ref{Claim:G as integral} is Theorem~\ref{thm:G}\ref{thm:G(i)}, and Claim~\ref{Claim:G as integral 2} is Theorem~\ref{thm:G}\ref{thm:G(ii)}.
\begin{claim}[Mahler] \label{Claim:G as integral}
We can represent $\Gzor$ as an $M$-fold integral as
	\begin{multline*}
    \vr\,! \cdot \Gzor = \\ (1-z)^{\omega_0} \int_0^z \int_0^{t_1} \int_0^{t_2} \cdots \int_0^{t_{M-1}}
	\mathcal{G}(z,t_1,t_2,\dots,t_M)\,dt_M\cdots dt_3\,dt_2\,dt_1,
    \end{multline*}
where
	\[
	\mathcal{G}(t_0,t_1,\dots,t_M)
		= t_M^{\rho_M}
			\left(\prod_{h=1}^{M} \left( \frac{t_{h-1}-t_{h}}{1-t_{h}}\right)^{\rho_{h-1}}\right)
			\left(\prod_{h=1}^M (1-t_h)^{\omega_h-\omega_{h-1}-1} \right).\]
\end{claim}

\begin{proof}[Mahler's Proof~\cite{Mahler}]
``This [Claim~\ref{Claim:G as DE}] can easily be brought to the following form.''
\end{proof}

This result allows one to produce to an efficient bound for $\Gzor$, and is thereby a lynchpin in applications. Other authors cite Mahler, or cite authors who cite Mahler. We did not find it easy, and hence indicate in some detail how to arrive at Mahler's conclusion.

\begin{proof}
We begin with the differential equation given in Claim~\ref{Claim:Differentials}:
    \[  (1-z)^{\omega_0+\rho_0+1}\left(\frac{d}{dz}\right)^{\rho_0+1} (1-z)^{-\omega_0} \Gzor 
	= \sG{z}{\langle \omega_1, \dots, \omega_M \rangle }{\langle \rho_1, \dots, \rho_M \rangle}. 
    \]
Hence,
    \begin{align*}
    \left(\frac{d}{dz}\right)^{\rho_0+1} (1-z)^{-\omega_0} \Gzor 
	&= (1-z)^{-(\omega_0+\rho_0+1)}
		\sG{z}{\langle \omega_1, \dots, \omega_M \rangle }{ \langle \rho_1, \dots, \rho_M \rangle} \\
	&= \sG{z}{\langle \omega_1, \dots, \omega_M\rangle-\omega_0-\rho_0-1}{\langle \rho_1, \dots, \rho_M \rangle},
    \end{align*}
where the second equality follows from Proposition~\ref{Claim:symmetry}.

We  observe that 
	\[
	\frac{d}{dz} \int_0^z \frac{(z-t)^k}{k!} f(t)\,dt 
		= \begin{cases} 
				f(z), & \text{$k=0$;} \\
				\int_0^z \frac{(z-t)^{k-1}}{(k-1)!} f(t)\,dt, & \text{$k>0$}.
		    \end{cases}
	\]
It then follows by repetition that
	\[
	\left(\frac{d}{dz} \right)^{\rho_0+1}\int_0^z \frac{(z-t)^{\rho_0}}{\rho_0!} f(t)\,dt = f(z).
	\]
Thus $(1-z)^{-\omega_0}\Gzor$ and $\int_0^z \frac{(z-t)^{\rho_0}}{\rho_0!} 
\sG{t}{\langle \omega_1, \dots, \omega_M\rangle-\omega_0-\rho_0-1}{\langle \rho_1, \dots, \rho_M \rangle}\, dt$ have the same $(\rho_0+1)$-th derivative. Therefore, they differ by a polynomial with degree at most $\rho_0$.

From the definition of $\Gzor$, we have
	\[
	\ord\left(\Gzor\right)
        = \rho_0+1+
            \ord\left(\sG{t}{\langle \omega_1, \dots, \omega_M\rangle-\omega_0-\rho_0-1}{\langle \rho_1, \dots, \rho_M \rangle} 
\right),
	\]
which dictates that the degree-at-most-$\rho_0$ polynomial is identically 0.
We can thus undo the differential operators as
	\[
	(1-z)^{-\omega_0}\Gzor
	= \int_0^z \frac{(z-t)^{\rho_0}}{\rho_0!} 
        \sG{t}{\langle \omega_1, \dots, \omega_M\rangle-\omega_0-\rho_0-1}{\langle \rho_1, \dots, \rho_M \rangle}\, dt,
	\]
whence
	\begin{equation}\label{equ:diff one step}
	\Gzor = \frac{(1-z)^{\omega_0}}{\rho_0!} \int_0^z (z-t)^{\rho_0} 
	\sG{t}{\langle \omega_1, \dots, \omega_M\rangle-\omega_0-\rho_0-1}{\langle \rho_1, \dots, \rho_M \rangle}\, dt.
	\end{equation}

We wish to apply equation~\eqref{equ:diff one step} inductively to express $\Gzor$ as an iterated integral. So that the notation will fit on the page, we define 
for $1\leq i\leq M$
	\begin{align*}
	S_0 &\coloneqq  0, \\
	S_i &\coloneqq  \omega_{i-1} + \rho_{i-1} + 1, \\
	G_0(z) &\coloneqq \Gzor, \\
	G_i(z) &\coloneqq \sG{z}{\langle \omega_i, \omega_{i+1},\dots,\omega_M \rangle -S_i }{ \langle \rho_i,\dots,\rho_M\rangle}.
	\end{align*}
Note that $G_M(z) = \frac{z^{\rho_M}}{\rho_M!}\, (1-z)^{\omega_M-S_M}$ by Claim~\ref{Claim:M=0}, while equation~\eqref{equ:diff one step} gives
	\[
	G_i(z) = \frac{(1-z)^{\omega_i-S_i}}{\rho_i!}  \int_0^z (z-t)^{\rho_{i}}G_{i+1}(t)\,dt
	\]
for $0\leq i < M$.
Now, iterating equation~\eqref{equ:diff one step} gives
	\begin{align*}
	\sG{t_0}{\vo}{\vr} &= G_0(t_0) \\
	  &= \frac{(1-t_0)^{\omega_0}}{\rho_0!} \int_0^{t_0} (t_0-t_1)^{\rho_0} G_1(t_1)\,dt_1 \\
	  &= \frac{(1-t_0)^{\omega_0}}{\rho_0!\rho_1!} \int_0^{t_0} (t_0-t_1)^{\rho_0} \cdot (1-t_1)^{\omega_1-S_1}\int_0^{t_1}(t_1-t_2)^{\rho_1} G_2(t_2)\,dt_2 \,dt_1 \\
	  &\,\,\,\vdots \\
	  &= \frac{(1-t_0)^{\omega_0}}{\rho_0!\cdots \rho_M!} \int_0^{t_0} \int_0^{t_1} \cdots \int_0^{t_{M-1}} \mathcal{G}(t_0,t_1,\dots,t_M)\,dt_M\cdots dt_2\,dt_1,\\
	\end{align*}
where
	\begin{align*}
	\mathcal{G}(t_0,t_1,\dots,t_M)
		&= \left(\prod_{h=0}^{M-1} (t_h-t_{h+1})^{\rho_h} \right)
			\left(\prod_{h=1}^{M-1} (1-t_h)^{\omega_h-S_h} \right) t_M^{\rho_M}(1-t_M)^{\omega_M-S_M}\\
		&= t_M^{\rho_M}
			\left(\prod_{h=1}^{M} \left( \frac{t_{h-1}-t_{h}}{1-t_{h}}\right)^{\rho_{h-1}}\right)
			\left(\prod_{h=1}^M (1-t_h)^{\omega_h-\omega_{h-1}-1} \right),
	\end{align*}
as claimed.
\end{proof}

\begin{claim}\label{Claim:G as integral 2}
The Padé remainder $\Gzor$ is given by the $M$-dimensional integral
        \[z^{\sigma-1}\, \frac{(1-z)^{\omega_0}}{\vr\,!} \int_{[0,1]^M}
                         \left( U_M^{-1}\prod_{h=1}^M U_h^{1+\rho_h} 
                         \left( \frac{1-u_h}{1-z U_h} \right)^{\rho_{h-1}} 
                         \left( 1- z U_h\right)^{\omega_h - \omega_{h-1}-1}\right)
                         \,d\vec u,
        \]
where $U_m = \prod_{h=1}^m u_h$.
\end{claim}

\begin{proof}
This follows from the previous claim upon the substitutions 
      \begin{equation*}
        t_h =z \, \prod_{i=1}^h u_i = z \, U_h, \qquad
        dt_M dt_{M-1} \cdots dt_2dt_1 = z^M \, \prod_{h=1}^{M-1} U_h d\vec u,
      \end{equation*}
and the obvious algebraic manipulations.
\end{proof}

\subsection{Contour Integrals and Derived Expressions}
Claim~\ref{Claim:G as contour integral} is Theorem~\ref{thm:G}\ref{thm:G(iii)}. Claim~\ref{Claim:H as a contour} is Theorem~\ref{thm:H}\ref{thm:H(i)}. Claim~\ref{Claim:H explicit} is Theorem~\ref{thm:H}\ref{thm:H(ii)}. Claim~\ref{Claim:H as gamma ratios} is Theorem~\ref{thm:H}\ref{thm:H(v)}. Claim~\ref{Claim:H as iterated integral} is Theorem~\ref{thm:H}\ref{thm:H(iii)}.
\begin{claim} \label{Claim:G as contour integral}
Let $\gamma$ be a simple positively oriented contour enclosing all $\sigma$ of the complex numbers $\omega_m+r$  ($0 \leq m \leq M, 0\leq r \leq 
\rho_m$). Then
	\[\Gzor = \frac{(-1)^{\sigma-1}}{2\pi i} \int_\gamma  (1-z)^\xi  \prod_{k=0}^M \frac{1}{\ff{(\xi-\omega_ k)}{\rho_ k+1}} \,d\xi.\]
\end{claim}

\begin{proof}
Set $$\sI{z}{\vo}{\vr} \coloneqq \frac{(-1)^{\sigma-1}}{2\pi i} \int_\gamma  (1-z)^\xi  \prod_{k=0}^M \frac{1}{\ff{(\xi-\omega_ k)}{\rho_ k+1}} \,d\xi,$$
and, as in Claim~\ref{Claim:G as DE},
    $$D_i\coloneqq  (1-z)^{\omega_i+\rho_i+1}\left(\frac{d}{dz}\right)^{\rho_i+1} (1-z)^{-\omega_i}.$$
We will show that 
    \[D_0 \sI{z}{\vo}{\vr}= \sI{z}{\vo^{\ast 0}}{\vo^{\ast 0}}.\]

Substituting yields
    \begin{multline*}
      D_0 \sI{z}{\vo}{\vr}
      =(1-z)^{\omega_0+\rho_0+1} \left(\frac{d}{dz}\right)^{\rho_0+1} (1-z)^{-\omega_0} \sI{z}{\vo}{\vr}\\
          =(1-z)^{\omega_0+\rho_0+1} \left(\frac{d}{dz}\right)^{\rho_0+1} 
                            \frac{(-1)^{\sigma-1}}{2\pi i} \int_\gamma  (1-z)^{\xi-\omega_0}  \prod_{k=0}^M \frac{1}{\ff{(\xi-\omega_ k)}{\rho_ k+1}} \,d\xi.
     \end{multline*}
As
    \[\left(\frac{d}{dz}\right)^{\rho_0+1} (1-z)^{\xi-\omega_0}  = (-1)^{\rho_0+1}
    \ff{(\xi-\omega_0)}{\rho_0+1}(1-z)^{\xi-\omega_0-\rho_0-1},\]
differentiating under the integral eliminates the $k=0$ factor in the product, giving
    \[D_0 \sI{z}{\vo}{\vr}
          =  \frac{(-1)^{\sigma-\rho_0-2}}{2\pi i}
                      \int_\gamma  (1-z)^{\xi} \prod_{k=1}^M \frac{1}{\ff{(\xi-\omega_ k)}{\rho_ k+1}} \,d\xi
          =\sI{z}{\vo^{\ast 0}}{\vr^{\ast 0}}.\]

We iterate, using $D_1,\dots,D_{M-1}$ successively to remove all but the  final coordinates of $\vo,\vr$, arriving at
    $$D_{M-1}\cdots D_1 D_0 \sI{z}{\vo}{\vr} 
    = \sI{z}{\langle \omega_M \rangle}{\langle \rho_M\rangle} =\frac{(-1)^{\rho_M}}{2\pi i} \int_\gamma  
     \frac{(1-z)^\xi }{\ff{(\xi-\omega_M)}{\rho_M+1}} \,d\xi .$$ 
By partial fractions~\cite{ConcreteMathematics}*{equation (5.41) in Section 5.3}
    \[ \frac{(-1)^{\rho_M}}{\ff{(\xi-\omega_M)}{\rho_M+1}}
        = \frac{1}{\rho_M!}\sum_{r=0}^{\rho_M} \frac{(-1)^r}{\xi-\omega_M-r}\binom{\rho_M}{r},\]
and with Cauchy's Integral Formula  we conclude
    \begin{align*}
    \frac{(-1)^{\rho_M}}{2\pi i} \int_\gamma  
         \frac{(1-z)^\xi }{\ff{(\xi-\omega_M)}{\rho_M+1}} \,d\xi 
      &=\frac{1}{2\pi i} \int_\gamma \frac{(1-z)^\xi}{\rho_M!}  \sum_{r=0}^{\rho_M} 
         \frac{(-1)^r}{\xi-\omega_M-r}\binom{\rho_M}{r}\,d\xi  \\     
      &=\frac{1}{\rho_M!} \sum_{r=0}^{\rho_M}\binom{\rho_M}{r} (-1)^r \frac{1}{2\pi i} \int_\gamma
         \frac{(1-z)^\xi}{\xi-\omega_M-r}\,d\xi  \\     
      &=\frac{1}{\rho_M!} \sum_{r=0}^{\rho_M}\binom{\rho_M}{r} (-1)^r  (1-z)^{\omega_M+r}  \\
      &=\frac{(1-z)^{\omega_M}}{\rho_M!} \sum_{r=0}^{\rho_M}\binom{\rho_M}{r}(z-1)^r  \\        
      &=\frac{(1-z)^{\omega_M}}{\rho_M!}\,z^{\rho_M}.
    \end{align*}
Thus,     $\sI{z}{\vo}{\vr} $ satisfies the DE in  Claim~\ref{Claim:G as DE}. We now show that it also satisfies the initial conditions given there, and so by Claim~\ref{Claim:G as DE} we will have $\sI{z}{\vo}{\vr} = \Gzor$.
 
As for the initial conditions, it remains to show that $\left. \frac{d^r}{dz^r} \sI{z}{\vo}{\vr}\right|_{z=0} = 0$ for $0\leq r \leq \sigma-2$, and for $r=\sigma-1$ we get $1$. We start with
        \[
          \frac{d^r}{dz^r} \sI{z}{\vo}{\vr}
            = \frac{(-1)^{\sigma-1}}{2\pi i} \int_\gamma  (-1)^r \ff{\xi}{r} (1-z)^{\xi-r}  \prod_{k=0}^M \frac{1}{\ff{(\xi-\omega_ k)}{\rho_ k+1}} \,d\xi \]
and evaluating this at $z=0$ gives
        \begin{equation} \label{equ: contour after diff}
         \frac{(-1)^{\sigma-r-1}}{2\pi i} \int_\gamma \frac{\ff{\xi}{r}}{\prod_{k=0}^M {\ff{(\xi-\omega_ k)}{\rho_ k+1}}} \, d\xi.
         \end{equation}
We may take $\gamma$ to be a circle with large radius $N$, where $N>\omega_k+r$ for $0\leq k \leq M$ and $0\leq r \leq \rho_k+1$. We now appeal to an argument that has little to do with our particular integrand, and so we generalize. Let $P(\xi)=\prod (\xi-p_j)$ be a monic polynomial of degree $r$, and let $Q(\xi)=\prod (\xi - q_j)$ be a monic polynomial of degree $\sigma$ with all of its roots inside $|\xi|=N$. Then, using the substitution $\xi \mapsto N^2/u$, which reverses the orientation of the contour,
	\begin{align*}
	\frac{1}{2\pi i} \int_{|\xi|=N} \frac{P(\xi)}{Q(\xi)}\,d\xi
	&= \frac{1}{2\pi i} \int_{|u|=N} \frac{P(N^2/u)}{Q(N^2/u)} \, \frac{N^2}{u^2}\, du \\
	&= \frac{N^2}{2\pi i} \int_{|u|=N} \frac{\prod (N^2/u - p_j)}{\prod (N^2/u-q_j)} \, \frac{du}{u^2} \\
	&= \frac{N^2}{2\pi i} \int_{|u|=N} \frac{ \prod (N^2 - u p_j)}{\prod (N^2 - u q_j)} u^{\sigma-r-2}\,du.
	\end{align*}
As all the roots of the denominator $\prod (N^2-u q_j)$ are outside the contour, this integral is 0 provided that $\sigma-r-2\geq 0$, that is, provided $r \leq \sigma-2$. If $r=\sigma-1$, then 
 	\begin{multline*}
	\frac{1}{2\pi i} \int_{|\xi|=N} \frac{P(\xi)}{Q(\xi)}\,d\xi
	= \frac{N^2}{2\pi i} \int_{|u|=N} \frac{ \prod (N^2 - u p_j)}{\prod (N^2 - u q_j)} u^{\sigma-r-2}\,du \\
	= N^2 \cdot \frac{\prod (N^2 - 0 \cdot p_j)}{\prod (N^2 - 0 \cdot q_j)} 
	= N^2 \cdot \frac{(N^2)^{r}}{(N^2)^\sigma} 
	=1.
	\end{multline*}
\end{proof}

\begin{claim} \label{Claim:H as a contour}
Let $\gamma_m$ be a simple positively oriented contour enclosing all $\rho_m+1$ of the complex numbers $\omega_m+r$ ($0\leq r \leq \rho_m$) and 
none of $\omega_k+r$  ($0 \leq k \leq M, k \neq m, 0\leq r \leq \rho_m$). Then
	\[
	\Hzor = \frac{(-1)^{\sigma-1}}{2\pi i} \int_{\gamma_m} (1-z)^{\xi-\omega_m}  \prod_{k=0}^M \frac{1}{\ff{(\xi-\omega_ k)}{\rho_ k+1}} \,d\xi .
	\]
\end{claim}

\begin{proof}
As no pair of the $\omega_i$ has a difference that is an integer, the $\sigma$ numbers $\omega_m+r$, where $0\leq m \leq M, 0\leq r \leq \rho_m$, are distinct.
Set
	\[
	\Phi_{r,m}(\xi) \coloneqq  (\xi-\omega_m-r)  \prod_{k=0}^M \frac{1}{\ff{(\xi-\omega_ k)}{\rho_ k+1}},
	\]
where we understand the removable singularity to be removed. Observe that each $\Phi_{r,m}$ has $\sigma-1$ simple poles.
We will evaluate $\Gzor$ using Cauchy's Integral Formula.
Let $\gamma_{r,m}$ be a simple closed contour enclosing $\omega_m+r$, but none of the roots of $\Phi_{r,m}$. From Claim~\ref{Claim:G as contour integral}, we find that
	\begin{align*}
	\Gzor
	&= \frac{(-1)^{\sigma-1}}{2\pi i} \int_\gamma  (1-z)^\xi  \prod_{k=0}^M \frac{1}{\ff{(\xi-\omega_ k)}{\rho_ k+1}} \,d\xi \\
	&= (-1)^{\sigma-1}\, 
		\sum_{m=0}^M \sum_{r=0}^{\rho_m} \frac{1}{2\pi i} \int_{\gamma_{r,m}}  \frac{(1-z)^\xi \Phi_{r,m}(\xi)}{\xi-\omega_m-r} \,d\xi \\
	&= (-1)^{\sigma-1}\, \sum_{m=0}^M \sum_{r=0}^{\rho_m} (1-z)^{\omega_m+r}\Phi_{r,m}(\omega_m+r) \\
	&=\sum_{m=0}^M  \left( (-1)^{\sigma-1} \sum_{r=0}^{\rho_m} (1-z)^r \, \Phi_{r,m}(\omega_m+r) \right) (1-z)^{\omega_m} .
	\end{align*}
We now notice that
	\begin{equation} \label{equ:Hzor as sum phi}
	 \Hzor =  (-1)^{\sigma-1}\, \sum_{r=0}^{\rho_m}  (1-z)^r \Phi_{r,m}(\omega_m+r),
	 \end{equation}
as this is a polynomial of the required degree.

Also,
	\begin{align*}
	\frac{(-1)^{\sigma-1}}{2\pi i} \int_{\gamma_m} &(1-z)^{\xi-\omega_m} \prod_{k=0}^M \frac{1}{\ff{(\xi-\omega_ k)}{\rho_ k+1}} \,d\xi \\
	&= (-1)^{\sigma-1} \sum_{r=0}^{\rho_m} \frac{1}{2\pi  i} \int_{\gamma_{m,r}} \frac{(1-z)^{\xi-\omega_m} \Phi_{r,m}(\xi)}{\xi-\omega_m-r}\,d\xi \\
	&= (-1)^{\sigma-1} \sum_{r=0}^{\rho_m} (1-z)^{(\omega_m+r)-\omega_m}\Phi_{r,m}(\omega_m+r) \\
	&= (-1)^{\sigma-1} \sum_{r=0}^{\rho_m} (1-z)^r \Phi_{r,m}(\omega_m+r) \\
	&= \Hzor.
	\end{align*}
\end{proof}

\begin{claim} \label{Claim:H explicit}
	\(\displaystyle \Hzor = \frac{1}{\rho_m!}\,\sum_{r=0}^{\rho_m} (z-1)^r \binom{\rho_m}{r} 
		\prod_{\substack{k=0\\ k\neq m}}^M \frac{1}{\rf{(\omega_k-\omega_m-r)}{\rho_ k+1}} . \)
\end{claim}

\begin{proof}
We continue with the notation of the proof of Claim~\ref{Claim:H as a contour}. In particular, we simplify the expression~\eqref{equ:Hzor as sum phi}. Observe that
	\[
	\Phi_{r,m}(\xi) = 
		\left[ \prod_{\substack{k=0\\ k\neq m}}^M \frac{1}{\ff{(\xi-\omega_ k)}{\rho_ k+1}} \right]
		\cdot \left[ \prod_{\substack{r'=0 \\ r' \neq  r}}^{\rho_m}  \frac{1}{\xi - \omega_m-r'} \right]
	\]
so that
	\[\Phi_{r,m}(\omega_m+r) 
	= \left[ \prod_{\substack{k=0\\ k\neq m}}^M \frac{1}{\ff{(\omega_m-\omega_ k+r)}{\rho_ k+1}} \right]
		\cdot \left[ \prod_{\substack{r'=0 \\ r' \neq  r}}^{\rho_m}  \frac{1}{r-r'} \right].
		\]
Now,
	\[
	\prod_{\substack{r'=0 \\ r' \neq  r}}^{\rho_m} (r-r') = (r)(r-1)\cdots (2)(1)(-1)(-2) \cdots (r-\rho_m) = (-1)^{r-\rho_m} \, r! \, (\rho_m-r)!,
	\]
so that
	\[\frac{1}{\prod_{\substack{r'=0 \\ r' \neq  r}}^{\rho_m} (r-r')} = \frac{(-1)^{r-\rho_m}}{\rho_m!} \, \binom{\rho_m}{r} .\]
Also,
	\begin{multline*}
	\prod_{\substack{k=0\\ k\neq m}}^M \ff{(\omega_m-\omega_ k+r)}{\rho_ k+1} 
	= \prod_{\substack{k=0\\ k\neq m}}^M (-1)^{\rho_k+1} \rf{(\omega_k-\omega_m-r)}{\rho_k+1}\\
	= (-1)^{\sigma-\rho_m-1}  \prod_{\substack{k=0\\ k\neq m}}^M \rf{(\omega_k-\omega_m-r)}{\rho_k+1}.
	\end{multline*}
We now have $\Hzor$ as claimed.
\end{proof}

\begin{claim} \label{Claim:H as gamma ratios}
Set $W\coloneqq W(m,k)=\omega_k-\omega_m$, and define $C_{m,k,r}$ by
\[C_{m,k,r} \coloneqq  \dbinom{\rho_k}{r},\]
if $m=k$, by
\[C_{m,k,r} \coloneqq
(-1)^{\rho_k+1} \dbinom{r}{\rho_k}^{-1} \dfrac{\Gamma(r+1)}{\Gamma(r+1-W)} \dfrac{\Gamma(r-\rho_k-W)}{\Gamma(r-\rho_k+1)}\]
if $m\neq k$ and $\rho_k<r$, and  by
\[C_{m,k,r} \coloneqq (-1)^r \dbinom{\rho_k}{r} \dfrac{\Gamma(r+1)}{\Gamma(r+1-W)} \dfrac{\Gamma(\rho_k-r+1)}{\Gamma(\rho_k-r+1+W)}\dfrac{\pi}{\sin(\pi W)}\]
if $m\neq k$ and $\rho_k \geq r$.
Then, we have
\[
\Hzor = \frac{1}{\vr\,!} \sum_{r=0}^{\rho_m} (z-1)^r \prod_{k=0}^M C_{m,k,r}.
\]
\end{claim}

\begin{proof}
We begin from Claim~\ref{Claim:H explicit}, writing $W$ in place of $\omega_k-\omega_m$:
\begin{align*}
\Hzor &= \frac{1}{\rho_m!} \sum_{r=0}^{\rho_m} (z-1)^r \binom{\rho_m}{r} \prod_{\substack{k=0 \\ k \neq m}}^M \frac{1}{\rf{(W-r)}{\rho_ k+1}} \\
  &= \frac{1}{\vr\,!} \sum_{r=0}^{\rho_m} (z-1)^r \binom{\rho_m}{r} \prod_{\substack{k=0 \\ k \neq m}}^M \frac{\rho_k!}{\rf{(W-r)}{\rho_ k+1}}.
\end{align*}

If $k\neq m$ and $\rho_k < r$, then 
  \[\rho_k! = \binom{r}{\rho_k}^{-1} \frac{r!}{(r-\rho_k)!} = \binom{r}{\rho_k}^{-1} \frac{\Gamma(r+1)}{\Gamma(r-\rho_k+1)} \] 
and 
  \[ \rf{(W-r)}{\rho_k+1} = (-1)^{\rho_k+1}\ff{(r-W)}{\rho_k+1} = (-1)^{\rho_k+1} \frac{\Gamma(r-W+1)}{\Gamma(r-W-\rho_k)}.\]
Combining these,
\begin{align*}
\frac{\rho_k!}{\rf{(W-r)}{\rho_k+1}}
  &= (-1)^{\rho_k+1} \binom{r}{\rho_k}^{-1} \frac{\Gamma(r+1)}{\Gamma(r-\rho_k+1)} \frac{\Gamma(r-W-\rho_k)}{\Gamma(r-W+1)} \\
  &= (-1)^{\rho_k+1} \binom{r}{\rho_k}^{-1} \frac{\Gamma(r+1)}{\Gamma(r+1-W)} \frac{\Gamma(r-\rho_k-W)}{\Gamma(r-\rho_k+1)} \\
  &= C_{m,k,r}.
\end{align*}

If $k\neq m$ and $\rho_k \geq r$, then
  \[ \rho_k! = \binom{\rho_k}{r} r! (\rho_k-r)! = \binom{\rho_k}{r} \Gamma(r+1)\Gamma(\rho_k-r+1)\]
and
\begin{align*}
    \rf{(W-r)}{\rho_k+1}
        &=\rf{(W-r)}{r} \cdot \rf{W}{\rho_k+1-r} \\
        &= (-1)^r \ff{(r-W)}{r} \cdot \ff{(W+\rho_k-r)}{\rho_k-r+1} \\
        &= (-1)^r \frac{\Gamma(r-W+1)}{\Gamma(r-W-r+1)} \cdot \frac{\Gamma(\rho_k-r+W+1)}{\Gamma(\rho_k-r+W-(\rho_k-r+1)+1)} \\
        &= (-1)^r \frac{\Gamma(r+1-W) \Gamma(\rho_k-r+1+W)}{\Gamma(1-W)\Gamma(W)}.
\end{align*}
By Euler's reflection formula for the Gamma fucntion, $\Gamma(1-W)\Gamma(W) = \pi/\sin(\pi W).$ Combining these,
\begin{align*}
\frac{\rho_k!}{\rf{(W-r)}{\rho_k+1}}
  &= (-1)^r \frac{\pi}{\sin(\pi W)} \frac{\Gamma(r+1)}{\Gamma(r+1-W)} \frac{\Gamma(\rho_k-r+1)}{\Gamma(\rho_k-r+1+W)} \binom{\rho_k}{r} \\
  &= C_{m,k,r}.
\end{align*}
This concludes the proof.
\end{proof}

\begin{claim} \label{Claim:H as iterated integral}
For $M\geq 1$, we can represent $\Hzor$ as an $M$-fold iterated (principal value) contour integral as
	\[
	\Hzor = \frac{Q_m}{\vr\,!} \, \int_{(G)} T_m^{-\omega_m-1} \bigg( \prod_{\substack{k=0\\ k\neq m}}^M  t_k^{\omega_k}(1+t_k)^{\rho_k}\bigg)
		\left(1 - (-1)^M \,\frac{1-z}{T_m} \right)^{\rho_m} \,d\vec{t},
	\]
where
	\begin{align*}
	Q_m 
		&\coloneqq  \prod_{\substack{k=0 \\ k \neq m}}^M \frac{1}{2i\sin(\pi(\omega_k-\omega_m))} \\
	T_m
		&\coloneqq  \prod_{\substack{k=0 \\ k \neq m}}^M t_k \\
	\int_{(G)} 
		&\coloneqq  \int_{|t_0|=1}  \cdots \int_{|t_{m-1}|=1}  \int_{|t_{m+1}|=1} \cdot \int_{|t_M|=1} \\
	d\vec{t}
		&\coloneqq  dt_M \cdots dt_{m+1}\,dt_{m-1} \cdots t_1.
	\end{align*}
\end{claim}

\begin{proof}
By induction and integration-by-parts, we notice that
	\begin{equation}\label{eq:PV integral}
	\PV \int_{|t|=1} t^{x-1} (1+t)^\rho\,dt = \int_{-\pi}^{\pi} e^{i(x-1)t}(1+e^{it})^{\rho} \,ie^{it}dt = \frac{2i\sin(\pi x) \, \rho!}{\rf{x}{\rho+1}},
	\end{equation}
provided that $\rho$ is a nonnegative integer and $x\in \CC\setminus\{0,-1,-2,\dots\}$. In this claim and its proof, all integrals are understood to be principal values.

Beginning with Claim~\ref{Claim:H explicit}, we may write $\Hzor$ as
\begin{align*} 
\textsc{P}& \textsc{oly}_m\!\big(z\, \big\lvert\,\substack{\vo \\ \vr} \big) 	\\
	&= \frac{1}{\rho_m!}\,\sum_{r=0}^{\rho_m} (z-1)^r \binom{\rho_m}{r} 
	\prod_{\substack{k=0\\ k\neq m}}^M \frac{1}{\rf{(\omega_k-\omega_m-r)}{\rho_ k+1}} \\
	&= \frac{1}{\vr\,!} \, \sum_{r=0}^{\rho_m} (z-1)^r  \binom{\rho_m}{r}
	\prod_{\substack{k=0\\ k\neq m}}^M \frac{1}{2i\sin(\pi(\omega_k-\omega_m-r))} 
	\frac{2i\sin(\pi (\omega_k-\omega_m-r))\, \rho_k!}{\rf{(\omega_k-\omega_m-r)}{\rho_ k+1}}.
\end{align*}
We now use equation~\eqref{eq:PV integral} to continue
\begin{align*}
\textsc{P}& \textsc{oly}_m\!\big(z\, \big\lvert\,\substack{\vo \\ \vr} \big) 	\\	&= 	 \frac{1}{\vr\,!}\, \sum_{r=0}^{\rho_m} (z-1)^r \binom{\rho_m}{r}
	\prod_{\substack{k=0\\ k\neq m}}^M (-1)^r Q_m \int_{|t_k|=1} t_k^{\omega_k-\omega_m-r-1}(1+t_k)^{\rho_k}\,dt_k \\
	&= \frac{Q_m}{\vr\,!} \, \sum_{r=0}^{\rho_m} (-1)^{rM} (z-1)^r \binom{\rho_m}{r} 
	\prod_{\substack{k=0\\ k\neq m}}^M  \int_{|t_k|=1} t_k^{\omega_k-\omega_m-r-1}(1+t_k)^{\rho_k}\,dt_k 
\end{align*}
Compressing the product of integrals using the $\int_{(G)}$ notation, we continue with
	\begin{align*}
\textsc{P}& \textsc{oly}_m\!\big(z\, \big\lvert\,\substack{\vo \\ \vr} \big) 	\\	&= \frac{Q_m}{\vr\,!} \, \int_{(G)} \sum_{r=0}^{\rho_m} (-1)^{rM} (z-1)^r \binom{\rho_m}{r}
		\prod_{\substack{k=0\\ k\neq m}}^M  t_k^{\omega_k-\omega_m-r-1}(1+t_k)^{\rho_k}\,d\vec{t} \\
	&= \frac{Q_m}{\vr\,!} \, \int_{(G)} \sum_{r=0}^{\rho_m} (-1)^{rM} (z-1)^r \binom{\rho_m}{r}
		T_m^{-r}\,T_m^{-\omega_m-1} \prod_{\substack{k=0\\ k\neq m}}^M  t_k^{\omega_k}(1+t_k)^{\rho_k}\,d\vec{t} \\
	&= \frac{Q_m}{\vr\,!} \, \int_{(G)} T_m^{-\omega_m-1}\bigg( \prod_{\substack{k=0\\ k\neq m}}^M  t_k^{\omega_k}(1+t_k)^{\rho_k}\bigg)
		\sum_{r=0}^{\rho_m} \binom{\rho_m}{r} \left((-1)^{M} \,\frac{z-1}{T_m}\right)^r \,d\vec{t} \\
	&= \frac{Q_m}{\vr\,!} \, \int_{(G)} T_m^{-\omega_m-1} \bigg( \prod_{\substack{k=0\\ k\neq m}}^M  t_k^{\omega_k}(1+t_k)^{\rho_k}\bigg)
		\left(1 - (-1)^M \,\frac{1-z}{T_m} \right)^{\rho_m} \,d\vec{t},
	\end{align*}
as asserted in the Claim.
\end{proof}

\subsection{Hypergeometric Functions}
Claim~\ref{Claim:G as Meijer} below is Theorem~\ref{thm:G}\ref{thm:G(v)}, and Claim~\ref{Claim:H as hypergeometric} is Theorem~\ref{thm:H}\ref{thm:H(iv)}.

The Meijer $G$-function~\cite{NISTonline} is defined for natural numbers $m,n,p,q$, provided $m\leq q$ and $n\leq p$, although we only encounter it here with 
$n=0, m=p=q=M+1$. It is denoted
\[\MeijerG{m}{n}{p}{q}{z}{a_1,a_2, \dots, a_p}{b_1,b_2,\dots,b_q}\]
and defined as
      \[ \frac{1}{2\pi i} \int_C \frac{\prod_{k=1}^m \Gamma(s+b_k) \prod_{k=1}^n \Gamma(1-a_k-s)}
            {\prod_{k=n+1}^p \Gamma(s+a_k) \prod_{k=m+1}^q \Gamma(1-b_k-s)} z^{-s}\,ds,\]
where $C$ is a particular infinite contour that separates the poles of $\Gamma(1-a_k-s)$ from those of $\Gamma(b_k+s)$; the particular contour required for convergence varies depending on $m,n,p,q,z$. 

\begin{claim} \label{Claim:G as Meijer}
$\Gzor$, when $|z|<1$ and $|1-z|<1$, is a special value of Meijer's $G$-function
   	\[ \Gzor = \MeijerG{M+1}{0}{M+1}{M+1}{1-z}{\vo+\vr+1}{\vo}.\]
\end{claim}

\begin{proof}[Sketch of Proof.]
With $m=p=q=M+1, n=0$, we see that 
	\[\prod_{k=1}^n \Gamma(1-a_k-s) = \prod_{k=m+1}^q \Gamma(1-b_k-s) =1.\]
Further, with $a_{k+1}=\omega_k+\rho_k+1, b_{k+1}=\omega_k$, 
	\[\frac{\prod_{k=1}^m \Gamma(s+b_k) }{\prod_{k=n+1}^p \Gamma(s+a_k)} 
		= \prod_{k=0}^M \frac{\Gamma(s+\omega_k)}{\Gamma(s+\omega_k+\rho_k+1)}
		= \prod_{k=0}^M \frac{1}{\rf{(s+\omega_k)}{\rho_k+1}}.
	\]
We now have
	\begin{align*}
	\MeijerG{M+1}{0}{M+1}{M+1}{1-z}{\vo+\vr+1}{\vo} 
		&= \frac{1}{2\pi i} \int_C  (1-z)^{-s} \, \prod_{k=0}^M\frac{1}{\rf{(s+\omega_k)}{\rho_k+1}} \,ds \\
		&= \frac{1}{2\pi i} \int_C  (1-z)^{\xi} \, \prod_{k=0}^M\frac{(-1)^{\rho_k+1}}{\ff{(\xi-\omega_k)}{\rho_k+1}} \,(-d\xi) \\
		&= \frac{(-1)^{\sigma-1}}{2\pi i} \int_{C}  (1-z)^{\xi} \, \prod_{k=0}^M\frac{1}{\ff{(\xi-\omega_k)}{\rho_k+1}} \,d\xi \\
		&= \Gzor.
	\end{align*}
Admittedly, we have played fast-and-loose with the contour, and therefore the conditions $|z|<1$ and $|1-z|<1$ are not explained.
\end{proof}

\begin{theorem}[Slater's Theorem~\cite{Slater}]\label{Slater}
Provided that $a_j-b_h$ is not a positive integer (with $j\leq n,h\leq m$), and $b_j-b_k$ is not an integer (with $1\leq j<k \leq q$), and $0<|z|<1$,
        \begin{multline*}
        \MeijerG{m}{n}{p}{q}{z}{a_1,a_2, \dots, a_p}{b_1,b_2,\dots,b_q} = \\
        \sum_{h=1}^m \frac{\prod_{\substack{k=1 \\ k \neq h}}^m \Gamma(b_k-b_h)\prod_{k=1}^n\Gamma(1+b_h-a_k)}
             {\prod_{k=m+1}^q \Gamma(1+b_h-b_k)\prod_{k=n+1}^p \Gamma(a_k-b_h)} \,
             \pFq{p}{q-1}{\vec{a}}{\vec{b}}{(-1)^{m+n-p}z}\,z^{b_h} ,
        \end{multline*}
where $\vec{a_h}=\langle 1+b_h-a_1, \dots,1+b_h-a_p \rangle$ and $\vec{b_h}=\langle 1+b_h-b_1,\dots,1+b_h-b_q \rangle$ (with the $1+b_h-b_h$ term 
omitted).
\end{theorem}

\begin{claim} \label{Claim:H as hypergeometric}
The Padé approximant $\Hzor$ is associated with a generalized hypergeometric function by
	\[
	\Hzor = \frac{1}{\rho_m!} \bigg( \prod_{\substack{k=0 \\ k \neq m}}^M \frac{1}{\rf{(\omega_k-\omega_m)}{\rho_k+1}}\bigg)
	\,\pFq{M+1}{M}{\omega_m-\vo-\vr}{(1+\omega_m - \vo)^{\star m}}{1-z}.
	\]
\end{claim}

\begin{proof}
Using Claim~\ref{Claim:G as Meijer} and Theorem~\ref{Slater}, we may write $\Gzor$ as 
\begin{equation*}
\sum_{m=0}^M \frac{\prod_{\substack{k=0 \\ k \neq m}}^M \Gamma(\omega_k-\omega_m)}{\prod_{k=0}^M \Gamma(\omega_k+\rho_k+1-\omega_m)}
		\,\pFq{M+1}{M}{1+\omega_m-\vo-\vr-1}{(1+\omega_m - \vo)^{\star m}}{1-z}
		\, (1-z)^{\omega_m},
\end{equation*}
which we manipulate into the form
\begin{equation*}
	\sum_{m=0}^M \frac{1}{\rho_m!} \bigg( \prod_{\substack{k=0 \\ k \neq m}}^M \frac{1}{\rf{(\omega_k-\omega_m)}{\rho_k+1}}\bigg)\,
		\,\pFq{M+1}{M}{\omega_m-\vo-\vr}{(1+\omega_m - \vo)^{\star m}}{1-z}
		\, (1-z)^{\omega_m}.
\end{equation*}
One coordinate of $\omega_m-\vo-\vr$ is $-\rho_m$, a nonpositive integer. Consequently,
	\[
	\frac{1}{\rho_m!} \bigg( \prod_{\substack{k=0 \\ k \neq m}}^M \frac{1}{\rf{(\omega_k-\omega_m)}{\rho_k+1}}\bigg)
	\,\pFq{M+1}{M}{\omega_m-\vo-\vr}{(1+\omega_m - \vo)^{\star m}}{1-z}
	\]
is a polynomial with degree at most $\rho_m$. Therefore, it must be $\Hzor$.
\end{proof}

\subsection{Power Series}
Claim~\ref{Claim:G as power series} is Theorem~\ref{thm:G}\ref{thm:G(iv)}.

\begin{claim} \label{Claim:G as power series}
Let $g_n$ be the coefficients in the power series expansion of $\Gzor$ at $z=0$, i.e., 
	\(\displaystyle \Gzor = \sum_{n=0}^\infty g_n \frac{z^n}{n!}.\)
Then for $n\ge  0$ we have
	\[g_n = (-1)^n \sum_{m=0}^M \frac{1}{\rho_m!} 
			\sum_{r=0}^{\rho_m} \binom{\rho_m}{r} \,
							\frac{(-1)^r \ff{(\omega_m+r)}{n}}{\prod_{\substack{k=0\\ k\neq m}}^M \rf{(\omega_k-\omega_m-r)}{\rho_k+1}}.
	\]
In particular, $g_n=0$ for $0\leq n \leq \sigma-2$ and $g_{\sigma-1}=1$.
\end{claim}

\begin{proof}
We begin with the contour integral representation of $\Gzor$ given in Claim~\ref{Claim:G as contour integral}, replace $(1-z)^\xi$ with its power series, and then integrate term by term, obtaining
\begin{align*}
	\Gzor &= \frac{(-1)^{\sigma-1}}{2\pi i} \int_\gamma  (1-z)^\xi  \prod_{k=0}^M \frac{1}{\ff{(\xi-\omega_ k)}{\rho_ k+1}} \,d\xi \\
		&= \frac{(-1)^{\sigma-1}}{2\pi i} \int_\gamma \left(\sum_{n=0}^\infty (-1)^n \ff{\xi}{n}\,\frac{z^n}{n!}\right) \prod_{k=0}^M \frac{1}{\ff{(\xi-\omega_ k)}{\rho_ k+1}} \,d\xi \\
		&= \sum_{n=0}^\infty (-1)^{n} \left(\frac{1}{2\pi i} \int_\gamma \frac{(-1)^{\sigma-1} \ff{\xi}{n}}{\prod_{k=0}^M \ff{(\xi-\omega_ k)}{\rho_ k+1}} \, d\xi \right) \, \frac{z^n}{n!}.
\end{align*}
Continuing as in the proof of Claims~\ref{Claim:H as a contour}  and~\ref{Claim:H explicit}, we see that
\begin{align}
	g_n 	&= (-1)^n \left(\frac{1}{2\pi i} \int_\gamma \frac{(-1)^{\sigma-1} \ff{\xi}{n}}{\prod_{k=0}^M \ff{(\xi-\omega_ k)}{\rho_ k+1}} \, d\xi \right) \label{equ:g_n first line}\\
		&= (-1)^n \sum_{m=0}^M \sum_{r=0}^{\rho_m} (-1)^{\sigma-1} \ff{(\omega_m+r)}{n} \Phi_{r,m}(\omega_m+r) \notag\\
		&= (-1)^n \sum_{m=0}^M \sum_{r=0}^{\rho_m} (-1)^{\sigma-1} \ff{(\omega_m+r)}{n} \frac{(-1)^{r-\rho_m}}{\rho_m!} \binom{\rho_m}{r}
				\prod_{\substack{k=0\\ k\neq m}}^M \frac{1}{\rf{(\omega_m+r-\omega_k)}{\rho_ k+1}} \notag\\
		&= (-1)^n \sum_{m=0}^M \sum_{r=0}^{\rho_m} (-1)^{\sigma-1} \ff{(\omega_m+r)}{n} \frac{(-1)^{r-\rho_m}}{\rho_m!} \binom{\rho_m}{r}
				\prod_{\substack{k=0\\ k\neq m}}^M \frac{(-1)^{\rho_k+1}}{\rf{(\omega_k-\omega_m-r)}{\rho_ k+1}} \notag\\
		&= (-1)^n \sum_{m=0}^M \frac{1}{\rho_m!} \sum_{r=0}^{\rho_m}  \binom{\rho_m}{r} \frac{(-1)^r \ff{(\omega_m+r)}{n}}
				{\prod_{\substack{k=0\\ k\neq m}}^M \rf{(\omega_k-\omega_m-r)}{\rho_ k+1}}.\notag
\end{align}
That $g_n=0$ for $0\leq n \leq \sigma-2$ and $g_{\sigma-1}=1$ follow from the definition of $\Gzor$. Alternatively, the expression on line~\eqref{equ:g_n first line} is shown directly to have these values in the proof of Claim~\ref{Claim:G as contour integral}, beginning with equation~\eqref{equ: contour after diff}.
\end{proof}

\subsection{Perfection}
We remind our reader that our vectors are indexed from 0, so that the $j$-th coordinate of $\langle \rho_0,\dots,\rho_M\rangle$ is $\rho_j$. The coordinates of the $(M+1) \times (M+1)$ matrix $\mathbf{H}$ in the next claim is indexed in the same manner.

Claim~\ref{Claim:perfect} is Theorem~\ref{thm:perfect}.

\begin{claim}\label{Claim:perfect}
Fix $\vr \in \NN^{M+1}$ and $\vec\epsilon_0, \vec\epsilon_1, \dots, \vec\epsilon_M \in \ZZ^{M+1}$ with each $\vr+\vec\epsilon_k$ having nonnegative coordinates, and denote the $j$-th coordinate of $\vec\epsilon_i$ as $\vec\epsilon_{i,j}$. Let $S$ be maximum of $\sum_{i=0}^M \vec\epsilon_{i,\beta(i)}$ taken over all permutations $\beta$ of $0,1, \dots, M$, and let $T$ be the minimum of $\sum_{j=0}^M \vec\epsilon_{i,j}$ taken over $0\leq i \leq M$. Suppose the following two conditions are satisfied:
\begin{enumerate}
  \item There is a unique permutation $\alpha$ of $0,1,\dots,M$ with $S=\sum_{i=0}^M \vec\epsilon_{i,\alpha(i)}$;
  \item $T+M = S$.
\end{enumerate}
Then the $(M+1) \times (M+1)$ matrix $\mathbf{H}$, whose $(k,m)$ coordinate is the polynomial $\sH{m}{z}{\vo}{\vr+\vec\epsilon_k}$, has determinant
  \[
  C z^{\sigma(\vr)+T-1},
  \]
where $C$ is nonzero and does not depend on $z$.
\end{claim}

\begin{proof}
The determinant of $\mathbf{H}$, by the familiar permutation expansion, is
	\[
	\det(\mathbf{H}) = \sum_{\beta \in \Sigma_{[0,M]}} (-1)^{\sgn(\beta)} \prod_{k=0}^M \sH{\beta(k)}{z}{\vo}{\vr+\vec\epsilon_k},
	\]
which is clearly a polynomial. Notice that
	\begin{multline*}
	\deg\left(\prod_{k=0}^M \sH{\beta(k)}{z}{\vo}{\vr+\vec\epsilon_k}\right) 
              = \sum_{k=0}^M \deg\left(\sH{\beta(k)}{z}{\vo}{\vr+\vec\epsilon_k}\right) \\
              = \sum_{k=0}^M ( \rho_{\beta(k)}+ \vec\epsilon_{k,\beta(k)})
              \leq \sigma(\vr)-(M+1)+S,
	\end{multline*}
with equality achieved for (and only for) $\beta=\alpha$. Consequently,     
    \[\deg(\det(\mathbf{H})) = \sigma(\vr)-M-1+S=\sigma(\vr)+T-1,\] 
and in particular $\det(\mathbf{H})$ is not identically 0. 

Let $\vec v$ be the column vector $\langle (1-z)^{\omega_0}, (1-z)^{\omega_1},\dots,(1-z)^{\omega_M}\rangle^T$. By definition $\mathbf{H}\vec v$ is a column of $M+1$ 
functions of $z$: in row $k$ it is $\sG{z}{\vo}{\vr+\vec\epsilon_k}$, which has a zero of order $\sigma(\vr+\vec\epsilon_k)-1 = \sigma(\vr)+\left(\sum_{j=0}^M \vec\epsilon_{k,j}\right) -1\geq \sigma(\vr) +T-1$. Now multiply $\mathbf{H} \vec v$ by the adjoint of $\mathbf{H}$, which is also a matrix of polynomials. We have
	\begin{align*}
	\det(\mathbf{H})\vec v
		&=\adj(\mathbf{H})\mathbf{H} \vec v \\
		&= \adj(\mathbf{H})
                  \begin{pmatrix}
                    \sG{z}{\vo}{\vr+\vec\epsilon_0} \\
                    \sG{z}{\vo}{\vr+\vec\epsilon_1} \\
                    \vdots \\
                    \sG{z}{\vo}{\vr+ \ve_M}
                  \end{pmatrix} \\
		&= \adj(\mathbf{H})
                  \left( z^{\sigma(\vr)+T-1} \sum_{n=0}^\infty \vec v_n z^n \right) \\
                &= z^{\sigma(\vr)+T-1} \sum_{n=0}^\infty (\adj(\mathbf{H})\vec v_n) z^n
	\end{align*}
for some column vectors $\vec v_0 \neq \vec0, \vec v_1, \dots$. That $\vec v_0 \neq \vec 0$ follows from the definition of $T$.
Each coordinate of $\det(\mathbf{H})\vec v$ has the form $\det(\mathbf{H})(1-z)^\omega$, and so has a zero at $z=0$ of order at most $\deg(\det(\mathbf{H}))= \sigma(\vr)-M-1+S$. By the above displayed equations, each coordinate of $\det(\mathbf{H}) \vec v$ has a zero at $z=0$ of order at least $\sigma(\vr)+T-1$ with equality for some coordinate. But $T+M=S$, by hypothesis, so that $\det(\mathbf{H})$ is a polynomial whose degree coincides with the order of its zero at $z=0$. Therefore
  \[\det(\mathbf{H})= C z^{\sigma(\vr)+T-1}= C z^{\sigma(\vr)+S-M-1},\]
as claimed. The constant $C$ is nonzero as $\det(\mathbf{H})$ is not identically~$0$.
\end{proof}

\section{Opportunities for Further Work}

\begin{enumerate}
\item Is there a nice iterated integral representation of $\Hzor$ without contours, similar to the representation in Theorem~\ref{thm:G}(i) for $\Gzor$?
\item For fixed $\vo$, which degree vectors $\vr^{(0)},\vr^{(1)},\dots,\vr^{(M)}$ lead to a perfect system? There seems to be some geometry involved. That is, a modest amount of computation suggests that for each $M$ there is $B$ such that if any coordinate of any $\vec\epsilon_k-\vec\epsilon_j$ is not between $-B$ and $B$, then the resulting system is not perfect for any $\rho$ (the determinant of $\mathbf{H}$ doesn't have the form $Cz^n$).
\item What is the value of $C$ in Theorem~\ref{thm:perfect}?
\item What is the nice power series expression for $\Hzor$? For $M=1$, this is an important part of the best explicit irrationality measure for $2^{1/3}$.
\end{enumerate}


\end{document}